\documentclass[11pt]{article}

\usepackage{amssymb} 
\usepackage{latexsym}
\usepackage{amsmath} 

\usepackage[utf8]{inputenc} 
\usepackage[dvips]{color}
\usepackage{multicol}

\def\n{\nabla}

\def\F{{\mathcal F}}
\def\I{{\mathcal I}}

\def\X#1{\mathfrak{X}(#1)}

\newenvironment{proof}{{\bf Proof.}}{\hfill$\rule{1ex}{1ex}$\par\medskip}
\newtheorem{theorem}{Theorem}[section]
\newtheorem{proposition}[theorem]{Proposition}

\newtheorem{remark}[theorem]{Remark}
\newtheorem{definition}[theorem]{Definition}
\newtheorem{lemma}[theorem]{Lemma}
\newtheorem{corollary}[theorem]{Corollary}
\newcommand{\bp}{\begin{proof}\;}
\newcommand{\ep}{\end{proof}}

\title{Finsler spaces with infinite dimensional holonomy group}

\author{Zolt\'an Muzsnay and P\'eter T. Nagy} \date{{\small Institute of
    Mathematics, University of Debrecen
    \\
    H-4010 Debrecen, Hungary, P.O.B. 12 \vspace{2pt}
    \\
    {\it E-mail}: {\tt {}muzsnay@math.unideb.hu}, {\tt
      {}nagypeti@math.unideb.hu}} \vspace{12pt}
  \\
  \emph{\small Dedicated to Professor Joseph Grifone on the occasion of
    his 70th birthday}}

\begin{document}

\maketitle

\footnotetext{2000 {\em Mathematics Subject Classification: 53B40,
    53C29, 22E65}.}  %
\footnotetext{{\em Key words and phrases:} Finsler geometry, holonomy,
  infinite-dimensional Lie groups.}%
\footnotetext{This research was supported by the Hungarian Scientific
  Research Fund (OTKA), Grant K 67617.}

\begin{abstract} 
  Our paper is devoted to the study of the holonomy groups of Finsler
  surfaces using the methods of infinite dimensional Lie theory. The
  notion of infinitesimal holonomy algebra will be introduced, by the
  smallest Lie algebra of vector fields on an indicatrix, containing the
  curvature vector fields and their horizontal covariant derivatives
  with respect to the Berwald connection. We obtain that the topological
  closure of the holonomy group contains the exponential image of any
  tangent Lie algebra of the holonomy group. A class of Randers surfaces
  is determined, for which the infinitesimal holonomy algebra coincides
  with the curvature algebra. We prove that for all projectively flat
  Randers surfaces of non-zero constant flag curvature the infinitesimal
  holonomy algebra has infinite dimension and hence the holonomy group
  cannot be a Lie group of finite dimension.  Finally, in the case of
  the Funk metric we prove that the infinitesimal holonomy algebra is a
  dense subalgebra of the Lie algebra of the full diffeomorphism group
  and hence the topological closure of the holonomy group is the
  orientation preserving diffeomorphism group of the circle.
\end{abstract}

\section{Introduction}

The holonomy group of a Finsler manifold is the subgroup of the
diffeomorphism group of an indicatrix, generated by canonical
homogeneous (nonlinear) parallel translations along closed loops. We
showed in a previous paper \cite{Mu_Na} that the holonomy group of a
Finsler manifold of non-zero constant flag curvature cannot be a
compact Lie group, and in general, the study its holonomy theory needs
the technique of infinite dimensional Lie groups. This paper is
devoted to the investigation of holonomy groups of Finsler manifolds
using the results of Omori's infinite dimensional Lie theory
\cite{Omori1, Omori2}.
\\
We prove that the topological closure of the holonomy group is
contained the exponential image of any tangent Lie algebra of the
holonomy group.  After the introduction of the notion of
\emph{infinitesimal holonomy algebra} as the smallest Lie algebra of
vector fields on an indicatrix, containing the curvature vector fields
and their horizontal covariant derivatives with respect to the Berwald
connection, we give a direct proof of the tangent property of this Lie
algebra, which is a result of M.~Crampin, D. J.~Saunders \cite{Cr_Sa}
obtained in the course of the discussion of what would be the natural
notion of holonomy algebra. We apply our results on the infinitesimal
holonomy algebra to the study of holonomy properties of
two-dimensional non-Riemannian Finsler manifolds of non-zero constant
flag curvature. We find a class of Randers surfaces for which the
infinitesimal holonomy algebra coincides with the curvature
algebra. We prove that for all projectively flat Randers surfaces of
non-zero constant flag curvature the infinitesimal holonomy algebra
has infinite dimension and hence the holonomy group cannot be a Lie
group of finite dimension.  Finally, we prove that in the case of Funk
metric the infinitesimal holonomy algebra is a dense subalgebra of the
Lie algebra of the full diffeomorphism group of the indicatrix,
containing the real Witt algebra of trigonometric polynomial vector
fields.  Hence the topological closure of this holonomy group is the
orientation preserving diffeomorphism group of the circle.

\section{Preliminaries}

Throughout this article, $M$ denotes a $C^\infty$ manifold, ${\mathfrak
  X}^{\infty}(M)$ denotes the vector space of smooth vector fields on
$M$ and ${\mathsf {Diff}}^\infty(M)$ denotes the group of all
$C^\infty$-diffeomorphism of $M$ with the $C^\infty$-topology.

If $M$ is a compact manifold then ${\mathsf {Diff}}^\infty(M)$ is a
$F$-regular infinite dimensional Lie group modeled on the vector space
${\mathfrak X}^{\infty}(M)$.  Particularly ${\mathsf {Diff}}^\infty(M)$
is a strong ILB-Lie group.  In this category of group one can define the
exponential mapping and the group structure is locally determined by the
Lie algebra by the exponential mapping.  The Lie algebra of
${\mathsf{Diff}^{\infty}}(M)$ is ${\mathfrak X}^{\infty}(M)$ equipped
with the negative of the usual Lie bracket (cf.~\cite{Omori1, Omori2}).

\subsubsection*{Finsler manifold, canonical connection, parallelism}

A \emph{Finsler manifold} is a pair $(M,\mathcal F)$, where $M$ is an
$n$-dimensional smooth manifold and $\F\colon TM \to \mathbb{R}$ is a continuous 
function, smooth on $\hat T M := TM\setminus\! \{0\}$, its 
restriction ${\mathcal F}_x={\mathcal F}|_{_{T_xM}}$  is a 
positively homogeneous function of degree $1$ and the symmetric bilinear form
\begin{displaymath}
  g_{x,y} \colon (u,v)\ \mapsto \ g_{ij}(x, y)u^iv^j=\frac{1}{2}
  \frac{\partial^2 \mathcal F^2_x(y+su+tv)}{\partial s\,\partial
    t}\Big|_{t=s=0}
\end{displaymath}
is positive definite at every $y\in \hat T_xM$.  \\[1ex]
\emph{Geodesics} of Finsler manifolds are determined by a
system of $2$nd order ordinary differential equation:
\begin{displaymath}
  \ddot{x}^i + 2 G^i(x,\dot x)=0, \quad i = 1,...,n
\end{displaymath}
where $G^i(x,\dot x)$ are locally given by
\begin{equation}
  \label{eq:G_i}  G^i(x,y):= \frac{1}{4}g^{il}(x,y)\Big(2\frac{\partial
    g_{jl}}{\partial x^k}(x,y) -\frac{\partial g_{jk}}{\partial
    x^l}(x,y) \Big) y^jy^k.
\end{equation}
The associated \emph{homogeneous (nonlinear) parallel translation} can be
defined as follows: a vector field $X(t)=X^i(t)\frac{\partial}{\partial
  x^i}$ along a curve $c(t)$ is said to be parallel if it satisfies
\begin{equation}
  \label{eq:D}
  D_{\dot c} X (t):=\Big(\frac{d X^i(t)}{d t}+  G^i_j(c(t),X(t))\dot c^j(t)
  \Big)\frac{\partial}{\partial x^i}
  =0,
\end{equation}
where $G^i_j=\frac{\partial G^i}{\partial y^j}$.

\subsubsection*{Horizontal distribution, Berwald connection, curvature}

Let $(TM,\pi ,M)$ and $(TTM,\tau ,TM)$ denote the first and the second
tangent bundle of the manifold $M$, respectively.  The horizontal
distribution ${\mathcal H}TM \!  \subset\! TTM$ associated to the
Finsler manifold $(M, \mathcal F)$ can be defined as the image of the
horizontal lift which is at each $x\in M$ an isomorphism
\begin{math}
  X \to X^h
\end{math}
between $T_xM$ and ${\mathcal H}_xTM$ defined by the formula
\begin{equation}
  \label{eq:lift}
  \Big(X^i\frac{\partial}{\partial x^i}\Big)^{\! h}:=
  X^i
  \left(
    \frac{\partial}{\partial x^i}
    -G_i^k(x,y)\frac{\partial}{\partial y^k}
  \right).
\end{equation}
If ${\mathcal V}TM:= \mathrm{Ker} \, \pi_{*} \!  \subset\!  TTM$
denotes the vertical distribution on $TM$, ${\mathcal V}_yTM:=
\mathrm{Ker} \, \pi_{*,y}$, then for any $y\in TM$ we have $T_yTM =
{\mathcal H}_yTM \oplus {\mathcal V}_yTM$. The projectors
corresponding to this decomposition will be denoted by $h:TTM \to
{\mathcal H}TM$ and $v:TTM \to {\mathcal V}TM$.  We note that the
vertical distribution is integrable.
\\[1ex]
Let $(\hat{\mathcal V}TM,\tau,\hat T M)$ be the vertical bundle over
$\hat T M := TM\setminus\! \{0\}$.  We denote by ${\mathfrak
  X}^{\infty}(M)$, respectively by $\hat{\mathfrak X}^{\infty}(TM)$
the vector space of smooth vector fields on $M$ and of smooth sections
of the bundle $(\hat{\mathcal V}TM,\tau,\hat T M)$, respectively. The
\emph{horizontal Berwald covariant derivative} of a section
$\xi\in\hat{\mathfrak X}^{\infty}(TM)$ by a vector field
$X\in{\mathfrak X}^{\infty}(M)$ is defined by
\begin{equation}
  \label{eq:berwald_h_v} \nabla_X\xi := [X^h,\xi].
\end{equation}
In an induced local coordinate system $(x^i,y^i)$ on $TM$ for the
vector fields $\xi(x,y) = \xi^i(x,y)\frac {\partial}{\partial y^i}$ and
$X(x) = X^i(x)\frac {\partial}{\partial x^i}$ we have (\ref{eq:lift})
and hence
\begin{equation}
  \label{covder}
  \nabla_X\xi = \left(\frac {\partial\xi^i(x,y)}{\partial x^j} 
    - G_j^k(x,y)\frac{\partial \xi^i(x,y)}{\partial y^k} + 
    G^i_{j k}(x,y)\xi^k(x,y)\right)X^j\frac {\partial}{\partial y^i},
\end{equation}
where \[G^i_{j k}(x,y) := \frac{\partial G_j^i(x,y)}{\partial y^k}.\]
Let $(\pi^{*}TM,\bar{\pi},\hat T M)$ be the pull-back bundle of $(\hat T
M,\pi ,M)$ by the map $\pi:TM\to M$.  Clearly, the mapping
\begin{equation}
 \label{iden}
 (x,y,\xi^i\frac {\partial}{\partial y^i})\mapsto(x,y,
 \xi^i\frac {\partial}{\partial x^i}):\;\hat{\mathcal V}TM
 \rightarrow \pi^{*}TM
\end{equation} 
is a canonical bundle isomorphism.  In the following we will use the
isomorphism (\ref{iden}) for the identification of these bundles. If
we define
\[\nabla_X  \phi= \left(\frac {\partial\phi}{\partial x^j} - 
  G_j^k(x,y)\frac{\partial \phi(x,y)}{\partial y^k}\right)X^j\] for a
smooth function $\phi:\hat TM\to \mathbb R$, the horizontal Berwald
covariant derivation (\ref{covder}) can be extended to the tensor bundle
over
$(\pi^{*}TM,\bar{\pi},\hat T M)$.\\[1ex]
The \emph{Riemannian curvature tensor} field characterizes the
integrability of the horizontal distribution: 
\begin{equation}
  \label{eq:R_1}
  R_{(x,y)}(X, Y):= v\big[X^h,Y^h\big], \qquad X, Y \in
  T_x M.
\end{equation}
If the horizontal distribution $\mathcal H TM$ is integrable, then the
Riemannian curvature is identically zero.  Using a local coordinate
system the expression of the Riemannian curvature tensor $R_{(x,y)} =
R^i_{jk}(x,y)dx^j\otimes dx^k\otimes\frac{\partial}{\partial x^i}$ on
the pull-back bundle $(\pi^{*}TM,\bar{\pi},\hat T M)$ is
\begin{displaymath}
  R^i_{jk}(x,y) =  \frac{\partial G^i_j(x,y)}{\partial x^k} -
  \frac{\partial G^i_k(x,y)}{\partial x^j} + G_j^m(x,y)
  G^i_{k m}(x,y) - G_k^m(x,y)
  G^i_{j m}(x,y). 
\end{displaymath} 
The manifold is called of constant flag curvature $\lambda\in{\mathbb R}$, if for any
$x\in M$ the local expression of the Riemannian curvature is
\begin{equation}\label{gorb}
  R^i_{jk}(x,y) = \lambda\left(\delta_k^ig_{jm}(x,y)y^m - \delta_j^ig_{km}(x,y)y^m
  \right).
\end{equation}
In this case the flag curvature of the Finsler manifold
(cf.~\cite{ChSh}, Section 2.1 pp. 43-46) does not depend either on the
point or on the 2-flag.\\[1ex]
The \emph{Berwald curvature tensor} field $B_{(x,y)} =
B^i_{jkl}(x,y)dx^j\otimes dx^k\otimes
dx^l\otimes\frac{\partial}{\partial x^i}$ is
\begin{equation}
  \label{berw}
  B^i_{jkl}(x,y) =  \frac{\partial G^i_{jk}(x,y)}{\partial y^l} 
  = \frac{\partial^3 G^i(x,y)}{\partial y^j\partial y^k\partial y^l}.
\end{equation}
The \emph{mean Berwald curvature tensor} field $E_{(x,y)} =
E_{jk}(x,y)dx^j\otimes dx^k$ is the trace
\begin{equation}
  \label{mberw}
  E_{jk}(x,y) = B^l_{jkl}(x,y) =
  \frac{\partial^3 G^l(x,y)}{\partial y^j\partial y^k\partial y^l}.
\end{equation}
The \emph{Landsberg curvature tensor} field $L_{(x,y)} =
L^i_{jkl}(x,y)dx^j\otimes dx^k\otimes
dx^l\otimes\frac{\partial}{\partial x^i}$ is
\[L_{(x,y)}(u,v,w)=g_{(x,y)}\left(\nabla_wB_{(x,y)}(u,v,w),y\right),
\quad u,v,w\in T_xM.\] According to Lemma 6.2.2, equation (6.30),
p. 85 in \cite {Shen1}, one has for $u,v,w\in T_xM$
\[\nabla_wg_{(x,y)}(u,v) = -2 L(u,v,w).\]
\begin{lemma} 
  \label{covcurv} 
  The horizontal Berwald covariant derivative of the tensor field
  \[Q_{(x,y)} = \left(\delta_j^ig_{km}(x,y)y^m -
    \delta_k^ig_{jm}(x,y)y^m\right) dx^j\otimes dx^k\otimes
  dx^l\otimes\frac{\partial}{\partial x^i}\] vanishes.
\end{lemma}
\begin{proof}
  For any vector field $W\in{\mathfrak X}^{\infty}(M)$ we have $\nabla_W
  y = 0$ and $\nabla_W \mathsf{Id}_{TM} = 0$. Moreover, since
  $L_{(x,y)}(y,v,w)= 0$ (cf. equation 6.28, p. 85 in \cite {Shen1}) we
  get the assertion.
\end{proof}

\section{Holonomy of Finsler spaces}

Let $(M,\mathcal F)$ be an $n$-dimensional Finsler manifold.  We
denote by $(\I M,\pi,M)$ the \emph {indicatrix bundle} of $(M,\mathcal
F)$, the \emph{indicatrix} $\I_xM$ at $x \in M$ is the compact
hypersurface 
\begin{displaymath}
  \I_xM:= \{y \in T_xM ; \ \mathcal F(x, y) = 1\},
\end{displaymath}
of $T_xM$ diffeomorphic to the standard $(n-1)$-sphere.
\\[1ex]
The homogeneous (nonlinear) parallel translation $\tau_{c}:T_{c(0)}M\to
T_{c(1)}M$ along a curve $c:[0,1]\to M$, defined by (\ref{eq:D})
preserves the value of the Finsler function, hence it induces a map
\begin{equation}
  \label{eq:tau_ind}
  \tau^\I_{c}\colon \I_{c(0)}M \longrightarrow \I_{c(1)}M
\end{equation}
between the indicatrices. 
\\[1ex]
The notion of the holonomy group of a Riemannian manifold can be
generalized very naturally for a Finsler manifold:
\begin{definition}
  \emph{The} holonomy group $\mathsf{Hol}(x)$ \emph{of a Finsler space
    $(M,\F)$ at a point $x\in M$ is the subgroup of the group of
    diffeomorphisms ${\mathsf{Diff}^{\infty}}({\I}_xM)$ of the
    indicatrix ${\I}_xM$ generated by (nonlinear) parallel translations
    of ${\I}_xM$ along piece-wise differentiable closed curves initiated
    at the point $x\!\in \!M$.}
\end{definition}
Clearly, the holonomy groups at different points of $M$ are
isomorphic.
\begin{definition}
  \emph{A vector field $\xi\in {\mathfrak X}^{\infty}({\I}M)$ on the
    indicatrix bundle $\I M$ is a \emph{curvature vector field} of the
    Finsler manifold $(M, \F)$, if there exist vector fields $X, Y\in
    {\mathfrak X}^{\infty}(M)$ on the manifold $M$ such that $\xi =
    r(X,Y)$, where for every $x\in M$ and $y \in {\I}_xM$ we have
    \begin{equation} \label{eq:r}
      r(X,Y)(x,y):=R_{(x,y)}(X_x,Y_x).
    \end{equation}
    If $x\in M$ is fixed and $X, Y\in T_xM$, then the vector field $y
    \to r(X,Y)(x,y)$ on ${\I}_x M$ is a \emph{curvature vector field
      at $x$} (see \cite{Mu_Na}).\\[1ex]
    The Lie algebra $\mathfrak{R}(M)$ of vector fields generated by the
    curvature vector fields of $(M, \F)$ is called the \emph{curvature
      algebra} of the Finsler manifold $(M, \F)$. For a fixed $x\in M$
    the Lie algebra $\mathfrak{R}_x$ of vector fields generated by the
    curvature vector fields at $x$ is called the \emph{curvature
      algebra at the point $x$}.}
\end{definition}
\begin{definition}
  \emph{The \emph{infinitesimal holonomy algebra} of the Finsler
    manifold $(M, \mathcal F)$ is the smallest Lie algebra
    $\mathfrak{hol}^{*}(M)$ of vector fields on the indicatrix bundle $\I
    M$ satisfying the properties
    \begin{enumerate}
    \item[(i)] any curvature vector field $\xi$ belongs to
      $\mathfrak{hol}^{*}(M)$,
    \item[(ii)] if $\xi, \eta\in \mathfrak{hol}^{*}(M)$ then 
      $[\xi,\eta]\in\mathfrak{hol}^{*}(M)$,  
    \item[(iii)] if $\xi\in\mathfrak{hol}^{*}(M)$ and $X\!\in\!
      {\mathfrak X}^{\infty}(M)$ then the horizontal Berwald covariant
      derivative $\nabla_{\!\! X}\xi$ also belongs to
      $\mathfrak{hol}^{*}(M)$.
    \end{enumerate} 
    The \emph{infinitesimal holonomy algebra at a point $x\in M$} is
    the Lie algebra
    \begin{displaymath}
      \mathfrak{hol}^{*}(x) \!  := \! \big\{ \, \xi(x) \
      ; \ \xi \in \mathfrak{hol}^{*}(M) \, \big \}\subset{\mathfrak
        X}^{\infty}({\I}_xM)
    \end{displaymath}
    of vector fields on the indicatrix $\I_x M$.} 
\end{definition} 
Clearly, $\mathfrak{R}(M)\subset\mathfrak{hol}^{*}(M)$ and
$\mathfrak{R}_x\subset\mathfrak{hol}^{*}(x)$ for any $x\!\in\! M$.

\section{Tangent Lie algebras to the holonomy group}

Before formulating the next theorem, we recall the notion of tangent
vector fields to a subgroup of the diffeomorphism group, introduced and
discussed in \cite{Mu_Na}. Let $H$ be a subgroup of the diffeomorphism
group $\mathsf{Diff}^{\infty}(M)$ of a differentiable manifold $M$ and
let
${\mathfrak X}^{\infty}(M)$ be the Lie algebra of smooth vector fields on $M$.
\begin{definition}
  \label{def:strongly}
  \emph{ A vector field $X\!\in\! {\mathfrak X}^{\infty}(M)$ is called
    \emph{tangent to} $H\subset\mathsf{Diff}^{\infty}(M)$ if there
    exists a ${\mathcal C}^1$-differentiable $1$-parameter family
    $\{\Phi(t)\!\in \!H\}_{t\in \mathbb R}$ of diffeomorphisms of $M$
    such that $\Phi(0)=\mathsf{Id}$ and \(
    \frac{\partial\Phi(t)}{\partial t}\big|_{t=0}\!=\!X.\)
    \\[1ex]
    A Lie subalgebra $\mathfrak g$ of ${\mathfrak X}^{\infty}(M)$ is
    called \emph{tangent to} $H$, if all elements of $\mathfrak g$ are
    tangent vector fields to $H$.}
\end{definition}
\begin{proposition} 
  \label{expo} 
  If the Lie subalgebra $\mathfrak g$ of ${\mathfrak X}^{\infty}(M)$
  is tangent to a closed subgroup $H$ of
  $\mathsf{Diff}^{\infty}(M)$, then the exponential image
  $\exp(\mathfrak g)$ of $\mathfrak g$ is contained in $H$.
\end{proposition}
\begin{proof} The diffeomorphism group $\mathsf{Diff}^{\infty}(M)$ is a
  strong $ILB$-Lie group (\cite{Omori2}, Theorem 2.1 in Ch. VI, p. 137)
  and hence it is a regular $F$-Lie group (\cite{Omori2}, Corollary 5.4,
  p. 84). According to the arguments used in the proof of this
  corollary, if $\{\Phi(s)\in H\}_{s\in\mathbb R}$ is a ${\mathcal
    C}^1$-differentiable $1$-parameter family of diffeomorphisms of $M$
  such that
  \begin{displaymath}
    \Phi(0)=\mathsf{Id} \quad \text{and} \quad 
    \frac{\partial\Phi(s)}{\partial s}\Big|_{s=0}=X,
  \end{displaymath}
  then the $1$-parameter families
  \begin{displaymath}
    \left\{\Phi\left(\frac{s}{n}\right)^n\in H\right\}_{s\in
      \mathbb R}, \quad n=1,2\dots
  \end{displaymath}
  of diffeomorphisms converge uniformly on each compact interval to the
  $1$-parameter group $\left\{\exp{sX}\right\}_{s\in \mathbb R}$, where
  $X\!\in\!\mathfrak g$. It follows that $\left\{\exp{sX}\right\}_{s\in
    \mathbb R}\!\subset\! H$ for any $X\in\mathfrak g$.
\end{proof}
\begin{theorem} 
  \label{thm:inf_hol_tan}
  The infinitesimal holonomy algebra $\mathfrak{hol}^{*}(x)$ at a point
  $x\!\in \!M$ has the following properties:
  \begin{enumerate}
  \item[{\em (i)}] $\mathfrak{hol}^{*}(x)$ is tangent to the holonomy
    group $\mathsf{Hol}(x)$,
  \item[{\em (ii)}] the group generated by the exponential image
    $\exp\!\big(\mathfrak{hol}^{*}(x)\big)$ is a subgroup of the
    topological closure of the holonomy group $\mathsf{Hol}(x)$.
  \end{enumerate}
\end{theorem}
To prove the theorem we have to introduce some notion. We say that a
vector field $\xi\!\in\! {\mathfrak X}^{\infty}(M)$ is \emph{strongly
  tangent} to $H\!\subset\!\mathsf{Diff}^{\infty}(M)$ if there exist $k
\!\in\!\mathbb N$ and a ${\mathcal C}^{\infty}$-differentiable
$k$-parameter family $\{\Phi_{(t_1,\dots,t_k)}\in H\}_{t_i\in
  (-\varepsilon,\varepsilon)}$ of diffeomorphisms associated to $\xi$
such that
\begin{enumerate}
\item \label{stronly_tan_1} $\Phi_{(t_1,\dots,t_k)}=\mathsf{Id}$, if
  $t_j=0$ for some $1\leq j\leq k;$
\item \label{stronly_tan_2}
  $\frac{\partial^k\Phi_{(t_1,\dots,t_k)}}{\partial t_1\cdots\partial
    t_k}\big|_{(t_1,\dots,t_k)=(0,\dots,0)}=\xi$.
\end{enumerate}
In \cite{Mu_Na} it was proved that if $\xi, \eta \!\in\!  {\mathfrak
  X}^{\infty}(M)$ are strongly tangent vector field to the group $H$,
then $[\xi,\eta]\in{\mathfrak X}^{\infty}(M)$ is also strongly tangent
to $H$. Moreover, if $\mathcal V$ is a set of strongly tangent
vector fields to the group $H$, then the Lie subalgebra $\mathfrak v$ of
${\mathfrak X}^{\infty}(M)$ generated by $\mathcal V$ is tangent to $H$.
\\
Now let $U$ be an open neighbourhood in the manifold $M$ and let
$\mathsf{Hol}_U(x)$ denote the holonomy group at $x\in U$ of the Finsler
submanifold $(U,\F|_U)$.  The group $\mathsf{Hol_f}(U)$ of fibre
preserving diffeomorphisms of the indicatrix bundle $(\I U,\pi|_U,U)$,
inducing elements of the holonomy group $\mathsf{Hol}_U(x)$ on any
indicatrix $\I_x$ will be called the \emph{fibred holonomy group} of the
submanifold $(U,\F|_U)$.
\\
It is clear that any strongly tangent vector field $\xi \!\in\!
{\mathfrak X}^{\infty}(\I U)$ to the fibred holonomy group
$\mathsf{Hol_f}(U)$ is a vertical vector field and for any $x\in U$ its
restriction $\xi_x:= \xi\big|_{\I_x}$ to the indicatrix $\I_x$ is
strongly tangent to the holonomy group $\mathsf{Hol}_U(x)$.
\begin{lemma} 
  \label{lemma_curv}
  If $U\!\subset\! M$ is diffeomorphic to $\mathbb{R}^n$, then any
  curvature vector field on $U$ is strongly tangent to the fibred
  holonomy group $\mathsf{Hol_f}(U)$.
\end{lemma}
\bp Since $U$ is diffeomorphic to $\mathbb{R}^n$, we can identify $U$
with the vector space $\mathbb{R}^n$.  Let $\xi\!=\!  r(X,Y)\!\in
\!{\mathfrak X}^{\infty}({\I}\mathbb{R}^n)$ be a curvature vector field,
where $X, Y\!\in\! {\mathfrak X}^{\infty}(\mathbb{R}^n)$. We show that
there exists a family
$\big\{\Phi_{(s,t)}\big|_{{\I}\mathbb{R}^n}\big\}_{s,t\in
  (-\varepsilon,\varepsilon)}$ of fibre preserving diffeomorphisms of
the indicatrix bundle $(\I \mathbb{R}^n,\pi,\mathbb{R}^n)$ such that for
any $x\!\in\! \mathbb{R}^n$ the induced family of diffeomorphisms of the
indicatrix ${\I}_x$ is contained in $\mathsf{Hol}_U(x)$ and
$\xi_x\!=\!\xi\big|_{{\I}_x\mathbb{R}^n}$ is the corresponding strongly
tangent vector field to $\mathsf{Hol}_U(x)$.
\\
For any $x\in\mathbb{R}^n$ and $0\leq s, t\leq 1$ let $\Pi(s X_x,t Y_x)$
be the parallelogram in $\mathbb{R}^n$ determined by the vertices $x$,
$x \!+\! s X_x$, $x \!+\! s X_x \!+\! t Y_x$, $x \!+\! t Y_x\!\in\!
\mathbb{R}^n$ and let $\tau _{\,\Pi(s X_x,t Y_x)}\!:\I_x M\to\I_x M$
denote the (nonlinear) parallel translation of the indicatrix $\I_x$
along the parallelogram $\Pi(s X_x,t Y_x)$. Clearly we have $\tau
_{\,\Pi(s X_x,t Y_x)}=\mathsf{Id}_{\I \mathbb{R}^n}$, if $s=0$ or $t=0$
and
\begin{displaymath}
  \frac{\partial^2\tau _{\,\Pi(s X_x,t Y_x)}} {\partial s\partial
    t}\Big|_{(s,t)=(0,0)} =\xi_x, \quad\text{for every}\quad x
  \in\mathbb{R}^n.
\end{displaymath}
Since $\Pi(s X_x,t Y_x)$ is a differentiable field of parallelograms in
$\mathbb{R}^n$, the maps $\tau _{\,\Pi(s X_x,t Y_x)}$ are fibre
preserving diffeomorphisms of the indicatrix bundle $\I\mathbb{R}^n$ for
any $0\leq s, t\leq 1$, and for any $x\in \mathbb{R}^n$ the induced
family of diffeomorphisms of the indicatrix ${\I}_x$ is contained in
$\mathsf{Hol}_U(x)$. Hence the vector field $\xi\!\in\! {\mathfrak
  X}^{\infty}(\mathbb{R}^n)$ is strongly tangent to the fibred holonomy
group $\mathsf{Hol_f}(U)$.  \ep
\begin{corollary} 
  \label{cvf} 
  If $U\!\subset\! M$ is diffeomorphic to $\mathbb{R}^n$, then the
  curvature algebra $\mathfrak{R}(U)$ is tangent to the fibred
  holonomy group $\mathsf{Hol_f}(U)$.
\end{corollary}

\begin{lemma} 
  \label{nabla_X} 
  If $\xi \!\in\! {\mathfrak X}^{\infty}({\I}U)$ is strongly tangent to
  the fibred holonomy group $\mathsf{Hol_f}(U)$ of $(U, \F|_U)$ then its
  horizontal covariant derivative $\nabla_{\!\! X}\xi$ with respect to
  any vector field $X\in{\mathfrak X}^{\infty}(U)$ is also strongly
  tangent to $\mathsf{Hol_f}(U)$.  Moreover, if $U$ is diffeomorphic to
  $\mathbb{R}^n$ then its infinitesimal holonomy algebra
  $\mathfrak{hol}^{*}(U)$ is tangent to the fibred holonomy group
  $\mathsf{Hol_f}(U)$.
\end{lemma}
\bp Let $\tau$ be the (nonlinear) parallel translation along the flow
$\varphi$ of the vector field $X$, i.e. for every $x\in U$ and $t \in
(-\varepsilon_x, \varepsilon_x)$ the map $\tau_t(x)\colon T_xU \to
T_{\varphi_t(x)}U$ is the (nonlinear) parallel translation along the
integral curve of $X$. If $\{\Phi_{(t_1,\dots,t_k)}\}_{t_i\in
  (-\varepsilon,\varepsilon)}$ is a ${\mathcal
  C}^{\infty}$-differentiable $k$-parameter family
$\{\Phi_{(t_1,\dots,t_k)}\}_{t_i\in (-\varepsilon,\varepsilon)}$ of
fibre preserving diffeomorphisms of the indicatrix bundle $(\I U,\pi
|_U,U)$ associated to the strongly tangent vector fields $\xi$ satisfying
the conditions 1.~and 2.~on page \pageref{stronly_tan_1}, then the
commutator
\begin{displaymath}
  [\Phi_{(t_1, \ldots ,t_k)}, \tau_{t_{k+1}}]:=
  \Phi^{-1}_{(t_1,...,t_k)}\circ \tau^{-1}_{t_{k+1}}\circ
  \Phi_{(t_1,...,t_k)}\circ \tau_{t_{k+1}}
\end{displaymath}
in the group $\mathsf{Diff}^{\infty}\big(\I U\big)$ fulfills
$[\Phi_{(t_1,...,t_k)},\tau_{t_{k+1}}]=\mathsf{Id}$, if some of its
variables equals $0$. Moreover
\begin{equation}
  \label{eq:phi_tau}
  \frac{\partial^{k+1}[\Phi_{(t_1...t_k)},\tau_{(t_{k+1})}]} {\partial
    t_1\;...\; \partial t_{k+1}} \Bigg|_{(0...0)}= \big[X^h, \xi\big]
\end{equation}
at any point of $U$, which shows that the vector field $\big[X^h,
\xi\big]$ is strongly tangent to $\mathsf{Hol_f}(U)$.  Moreover, since
the vector field $\xi$ is vertical, we have $h[X^h, \xi]=0$, and using
(\ref{eq:berwald_h_v}) we obtain
\begin{displaymath} 
  [X^h, \xi] =v[X^h, \xi] = \n_{X}\xi
\end{displaymath}
which yields to the first part of the assertion.  Moreover, with Lemma
\ref{lemma_curv} we obtain that the generating elements of
$\mathfrak{hol}^{*}(U)$ are strongly tangent to $\mathsf{Hol_f}(U)$,
therefore $\mathfrak{hol}^{*}(U)$ is tangent $\mathsf{Hol_f}(U)$.  \ep
\textbf{Proof} of Theorem \ref{thm:inf_hol_tan}.  Assertion (i) follows
from the previous lemma. Assertion (ii) is a consequence of Proposition
\ref{expo}.  \hfill$\rule{1ex}{1ex}$\par\medskip

\section{Holonomy algebras of Finsler surfaces}

We call a Finsler manifold $(M, \F)$ Fisler surface if $\dim M = 2$.
Let $M$ be diffeomorphic to $\mathbb{R}^2$ and let $X, Y\in {\mathfrak
  X}^{\infty}(M)$ be nonvanishing vector fields such that for any $x\in
M$ the values $X_x, Y_x\in T_xM$ are linarly independent. We denote by
$\xi\!\in\!\X{\I M}$ the curvature vector field
\begin{math}
  \xi \!=\!  r(X,Y)(x,y)\!=\!R_{(x,y)}(X_x,Y_x).
\end{math}
Since the indicatrix is 1-dimensional, any curvature vector field $\eta$
can be written as $\eta_x \!=\! f(x)\xi_x$, where the factor $f(x)$ is
arbitrary smooth function on $M$.  Therefore the commutators of
curvature vector fields are trivial and the curvature algebra
$\mathfrak{R}(M)$ is commutative.  Particularly, the curvature algebra
$\mathfrak{R}_x(M)$ at any point $x\!\in\! M$ is generated by the
vector field $\xi_x \!\in\!\X{\I_x M} $ and hence it is at most
1-dimensional.
\\
Even in this case, the infinitesimal holonomy algebra
$\mathfrak{hol}^{*}(x)$ at a point $x\!\in\! M$ (and hence the
corresponding holonomy group $\mathsf{Hol}_x(M)$) can be higher -- even
infinite -- dimensional. To show this we use a classical result of
S. Lie on the classification of Lie group actions on one-manifolds
(cf.~\cite{Brou} or \cite{Olver},
pp. 58-62):\\[1ex]
{\it If a finite-dimensional connected Lie group acts on a
  $1$-dimensional manifold without fixed points, than its dimension is
  less than $4$.}
\begin{proposition}\label{sec:prop-1}
  If the infinitesimal holonomy algebra $\mathfrak{hol}^{*}(x)$ contains
  $4$ simultanously non-vanishing $\mathbb R$-linearly independent
  vector fields, then the holonomy group $\mathsf{Hol}_M(x)$ is not a
  finite-dimensional Lie group.
\end{proposition}
\begin{proof} Indeed, in this case the holonomy group acts on the
  1-dimensional indicatrix without fixed points.  If it would be
  finite-dimensional then its dimension, and hence the dimension of its
  Lie algebra should be less than $4$.  This is a contradiction.
\end{proof}

\subsection{Finsler surfaces of constant flag curvature}

The relation between the infinitesimal holonomy algebra and the
curvature algebra is enlightened by the following
\begin{theorem} 
  \label{berwald}
  Let $(M, \F)$ be a Fisler surface with non-zero constant flag
  curvature. The infinitesimal holonomy algebra $\mathfrak{hol}^{*}(x)$
  at a point $x\in M$ coincides with the curvature algebra
  $\mathfrak{R}_x$ at $x$ if and only if the mean Berwald curvature
  $E_{(x,y)}$ of $(M, \F)$ vanishes for any $y\in {\I}_xM$.
\end{theorem}
\begin{proof} Let $U\!\subset\!M$ be a neighbourhood of $x\!\in\! M$
  diffeomorphic to $\mathbb{R}^2$.  Identifying $U$ with $\mathbb{R}^2$
  and considering a coordinate system $(x_1,x_2)$ in $\mathbb{R}^2$ we
  can write
  \begin{displaymath}
    R^i_{jk}(x,y) = \lambda\left(\delta_j^ig_{km}(x,y)y^m - 
      \delta_k^ig_{jm}(x,y)y^m \right), \quad   \text{with}\quad \lambda\neq 0.
  \end{displaymath}
  Since the curvature tensor field is skew-symmetric, $R_{(x,y)}$ acts
  on the one-dimensional wedge product $T_xM\wedge T_xM$. According to
  Lemma \ref{covcurv} the covariant derivative of the curvature
  vector field
  \begin{math}
    \xi \!= \!R(X,Y)\!=\!\frac12 R(X\otimes Y \!-\! Y\otimes X) \!=\! R(X\wedge Y)
  \end{math}
  can be written in the form
  \begin{displaymath}
    \n_Z\xi = \n_Z\left(r(X,Y)\right)=R\left(\n_Z(X\wedge Y)\right) = 
    R(\n_ZX\wedge Y+X\wedge \n_ZY),
  \end{displaymath}
  where $X, Y, Z\!\in\! \X{U}$. If $X\!=\!X^i\frac{\partial}{\partial
    x^i}$, $Y\!=\!Y^i\frac{\partial}{\partial x^i}$ and
  $Z\!=\!Z^i\frac{\partial}{\partial x^i}$ then we have $X\wedge Y =
  \frac12\left(X^1Y^2-X^2Y^1\right)\frac{\partial}{\partial x^1}
  \wedge\frac{\partial}{\partial x^2}$ and hence we obtain
  \begin{equation}
    \label{dercurv}
    \n_Z\xi =  R\left(\n_k\left((X^1Y^2-Y^1X^2)\frac{\partial}{\partial x^1}
        \wedge\frac{\partial}{\partial x^2}\right)Z^k\right)=
  \end{equation}
  \begin{displaymath}
    = R\left(\frac{\partial(X^1Y^2-Y^1X^2)} {\partial
        x^k} Z^k \frac{\partial}{\partial x^1}
      \wedge\frac{\partial}{\partial x^2}\right) +
    (X^1Y^2-Y^1X^2)R\left(\n_k\left(\frac{\partial}{\partial x^1}
        \wedge\frac{\partial}{\partial x^2}\right)\right)Z^k,
  \end{displaymath}
  where we denote the covariant derivative $\n_Z$ by $\n_k$ if $Z \!=\!
  \frac {\partial}{\partial x^k}$, $k\!=\!1,2$.  For given vector fields
  $X, Y, Z \in {\mathfrak X}^{\infty}(U)$ the expression
  $\frac{\partial(X^1Y^2-Y^1X^2)}{\partial x^k}Z^k$ is a function on
  $U$. Hence there exists a function $\psi$ on $U$ such that
  \[R\left(\frac{\partial(X^jY^h-Y^jX^h)}{\partial x^k}Z^k
    \frac{\partial}{\partial x^j}\wedge\frac{\partial}{\partial
      x^h}\right) = \psi\, R(X\wedge Y) = \psi\, R(X,Y),\] 
  and
  $\psi R(X,Y)$ is an element of the curvature algebra
  $\mathfrak{R}(U)$ of the submanifold $(U, \F|_U)$. \\[1ex]
  Now, we investigate the second term of the right hand side of
  (\ref{dercurv}).
  \begin{alignat*}{1}
    \n_k\Big(\frac{\partial}{\partial x^1}
    \wedge\frac{\partial}{\partial x^2}\Big)&=
    \Big(\n_k\frac{\partial}{\partial x^1}\Big)
    \wedge\frac{\partial}{\partial x^2} + \frac{\partial}{\partial x^1}
    \wedge\Big(\n_k\frac{\partial}{\partial x^2}\Big) =
    \\
    & = G^l_{k1}\frac{\partial}{\partial x^l}
    \wedge\frac{\partial}{\partial x^2} + \frac{\partial}{\partial x^1}
    \wedge G^m_{k2}\frac{\partial}{\partial x^m} = \left(G^1_{k1} +
      G^2_{k2}\right)\frac{\partial}{\partial x^1}
    \wedge\frac{\partial}{\partial x^2}.
  \end{alignat*}
  Hence 
  \begin{displaymath}
    (X^1Y^2\!-\!Y^1X^2)R\Big(\n_k\Big(\frac{\partial}{\partial x^1}
    \wedge\frac{\partial}{\partial x^2}\Big)\Big)Z^k \!=\!
    \left(G^1_{k1} \!+\! G^2_{k2}\right)Z^k R(X,Y)
    \!=\! \left(G^1_{k1} + G^2_{k2}\right)Z^k\xi
  \end{displaymath}
  This expression belongs to the curvature algebra if and only if the
  function $G^1_{k1} + G^2_{k2}$ does not depend on the variable $y$,
  i.e. if and only if
  \[E_{kh} = \frac{\partial\left(G^1_{k1} + G^2_{k2}\right)}{\partial
    y^h} = 0, \quad h, k =1, 2,\] identically.
\end{proof}

\begin{remark} 
  \label{covxi} 
  \emph{Let $\xi \!=\! R(X,Y)$ be a curvature vector field. Assume that
    the vector fields $X, Y\!\in\! {\mathfrak X}^{\infty}(M)$ have
    constant coordinate functions in a local coordinate system
    $(x^1,...,x^n)$ of the Finsler surface $(M, \F)$. Then we have in
    this coordinate system}
  \[\n_Z\xi = \left(G^1_{k1} + G^2_{k2}\right)Z^k\xi .\]
\end{remark}

\subsection{Randers surfaces with $\mathfrak{hol}^{*}(x)=
  \mathfrak{R}_x$}

A Fisler manifold $(M,\F)$ is called \emph{Randers manifold} if its
Finsler function has the form $\F = \alpha + \beta$, where $\alpha =
\sqrt{\alpha_{jk}(x)y^jy^k}$ is a Riemannian metric and $\beta =
\beta_j(x)y^j$ is a linear form. Z.~Shen constructed in \cite{Shen2}
families of Randers surfaces depending on the real parameter $\epsilon$,
which are of constant flag curvature $1$ on the unit sphere
$S^2\subset{\mathbb R}^3$ and of constant flag curvature $-1$ on a disk
${\mathbb D}^2\subset{\mathbb R}^2$.  These Finsler surfaces are not
projectively flat and have vanishing $S$-curvature (c.f.~\cite{Shen2},
Theorems 1.1 and 1.2). Their Finsler function is defined by
\begin{equation}
  \label{surface} 
  \alpha = \frac{\sqrt{\epsilon^2h(v,y)^2 + h(y,y)\left(1 - \epsilon^2h(v,v)\right)}}
  {1 - \epsilon^2h(v,v)}, \quad \beta = \frac{\epsilon h(v,y)}
  {1 - \epsilon^2h(v,v)}, 
\end{equation}
where $h(v,y)$ is the standard metric of the sphere $S^2$, respectively
$h(v,y)$ is the standard Klein metric on the unit disk ${\mathbb D}^2$
and $v$ denotes the vector field defined by $(-x_2,x_1,0)$ at
$(x_1,x_2,x_3)\in S^2$, respectively by $(-x_2,x_1)$ at $(x_1,x_2)\in
{\mathbb D}^2$.
\begin{theorem} 
  For any Randers surface defined by (\ref{surface}) the infinitesimal
  holonomy algebra $\mathfrak{hol}^{*}(x)$ at a point $x\in M$ coincides
  with the curvature algebra $\mathfrak{R}_x$.
\end{theorem}
\bp According to Theorem 1.1 and 1.2 in \cite{Shen2}, the above
classes of not locally projectively flat Randers surfaces with
non-zero constant flag curvature have vanishing S-curvature.
Moreover, Proposition 6.1.3 in \cite{Shen1}, p.~80, states that the
mean Berwald curvature vanishes if and only if the S-curvature is a
linear form on the surface. Hence the assertion follows from 
Corollary \ref{berwald}. \ep

\subsection{Randers surfaces with infinite dimensional
  $\mathfrak{hol}^{*}(x)$}

Projectively flat Randers manifolds with constant flag curvature
were classified by Z. Shen in \cite{Shen3}. He proved that any
projectively flat Randers manifold with non-zero constant flag curvature 
has negative curvature. This metric can be normalized by a constant factor  
so that the curvature is $-\frac14$. In this case it is isometric to the 
Randers manifold $(M,\F)$ defined by $\F = \alpha + \beta$ on the unit 
ball $\mathbb D^2 \subset \mathbb R^2$, where 
\begin{equation}
  \label{projective} 
  \alpha = \frac{\sqrt{|y|^2 - \left(|x|^2|y|^2 - 
        \langle x,y\rangle^2\right)}}{1 - |x|^2}, \quad 
  \beta = \frac{\langle x,y\rangle}{1 - |x|^2} + 
  \frac{\langle a,y\rangle}{1 - \langle a,x\rangle} 
\end{equation}
and $a\in \mathbb R^2$ is any constant vector with $|a|<1$.
\begin{theorem}
  Let $(M, \F)$ be a locally projectively flat Randers surface of
  non-zero constant flag curvature. There exists a point $x\!\in\! M$
  such that the infinitesimal holonomy algebra $\mathfrak{hol}^{*}(x)$
  is an infinite dimensional Lie algebra and hence the holonomy group of
  $(M, \F)$ is not a finite-dimensional Lie group.
\end{theorem}
\bp Since the metrics $\F \!=\! \alpha \!+\! \beta$ defined on the unit
ball $\mathbb D^2 \subset \mathbb R^2$ by (\ref{projective}) are
projectively flat, the geodesic coefficients (\ref{eq:G_i}) are of the
form $G^i(x,y) = P(x,y)y^i$, and hence
\begin{displaymath}
  G^i_k = \frac{\partial P}{\partial y^k}y^i + P\delta^i_k,\quad
  G^i_{kl} = \frac{\partial^2 P}{\partial y^k\partial y^l}y^i +
  \frac{\partial P}{\partial y^k}\delta^i_l + \frac{\partial P}{\partial
    y^l}\delta^i_k, \quad G^m_{km} = (n + 1)\frac{\partial P}{\partial
    y^k}. 
\end{displaymath}
Let us choose $x\!=\!0\in\mathbb D^2 \subset \mathbb R^2$. According to
\cite{Shen4}, eq. (41) and (42), pp. 1722-1723, the function $P(x,y)$
has the form
\begin{equation}
  \label{Pexpr}
  P(x,y) = \frac12\left\{\frac{\sqrt{|y|^2 - 
        \left(|x|^2|y|^2 - \langle x,y\rangle^2\right)} 
      + \langle x,y\rangle}{1 - |x|^2} - \frac{\langle a,y\rangle}
    {1 - \langle a,x\rangle}\right\}.
\end{equation}
Using Remark \ref{covxi} we obtain
\begin{equation}
  \label{n1xi}\n_Y\xi = G^m_{km}Y^k\xi = 
  3\frac{\partial P}{\partial y^k}Y^k\xi.
\end{equation}
Hence
\[\n_X\left(\n_Y\xi\right) = 3\n_X\left(\frac{\partial P}
  {\partial y^k}Y^k\xi\right) = 3\left\{\n_X\left(\frac{\partial P}
    {\partial y^k}Y^k\right)\xi + \left(\frac{\partial P}
    {\partial y^k}Y^k\right)\left(\frac{\partial P}
    {\partial y^l}X^l\right)\right\}\xi.\]
Assume that the vector field $Y$ has constant coordinate functions. 
Then we can write 
\begin{displaymath}
  \nabla_X \left(\frac{\partial P}{\partial y^k}Y^k\right)= 
  \left(\frac {\partial^2 P}{\partial x^j\partial y^k} -
    G_j^k\frac{\partial^2 P}{\partial y^k\partial y^k}\right)Y^kX^j =
\end{displaymath}
\begin{displaymath}
  = \left(\frac {\partial^2 P}{\partial x^j\partial y^k} -
    \Big(\frac{\partial P}{\partial y^j}y^m + P\delta^m_j\Big)
    \frac{\partial^2 P}{\partial y^k\partial y^m}\right)Y^kX^j =
  \left(\frac {\partial^2 P}{\partial x^j\partial y^k} - P\frac{\partial^2
      P}{\partial y^k\partial y^j}\right)Y^kX^j.
\end{displaymath}
It follows that
\begin{equation}
  \label{n2xi}
  \n_X\!\left(\n_Y\xi\right) = 
  3\left\{\frac {\partial^2 P}{\partial x^j\partial y^k} - 
    P\frac{\partial^2 P}{\partial y^k\partial y^j} + \frac{\partial P}
    {\partial y^k}\frac{\partial P}{\partial y^l}\right\}Y^kX^l\xi.\end{equation}
We want to prove that the vector fields $\xi\big|_{x=0}$, $\n_1\xi|_{x=0}$, 
$\n_2\xi|_{x=0}$ and $\n_1\left(\n_2\xi\right)|_{x=0}$ are linearly 
independent. We obtain from equations (\ref{n1xi}) and (\ref{n2xi}) that 
it is sufficient to show that the functions 
\begin{equation}
  \label{indep}
  1,\quad \frac{\partial P}{\partial y^1}\Big|_{x=0},\quad 
  \frac{\partial P}{\partial y^2}\Big|_{x=0}\quad 
  \text{and}\quad  \left(\frac {\partial^2 P}{\partial x^1\partial y^2} - 
    P\frac{\partial^2 P}{\partial y^1\partial y^2} + \frac{\partial P}
    {\partial y^1}\frac{\partial P}{\partial y^2}\right)\Big|_{x=0}\end{equation}
are linearly independent. We obtain from (\ref{Pexpr}) that 
\begin{displaymath}
  \ P\Big|_{x=0} = \frac12\left(|y| - \langle a,y\rangle\right), \quad  \frac{\partial
    P}{\partial y^k}\Big|_{x=0} =  \frac12\left(\frac{y^k}{|y|} -
    a^k\right), \quad \frac{\partial^2 P}{\partial y^j\partial
    y^k}\Big|_{x=0} =  \frac{1}{2|y|}\left(\delta^j_k -
    \frac{y^jy^k}{|y|^2}\right)
\end{displaymath}
and
\begin{displaymath}
  \frac{\partial P}{\partial x^j}\Big|_{x=0} = 
  \frac12\left(y^j - \langle a,y\rangle a^j\right),\quad \quad
  \frac{\partial^2 P}{\partial x^j \partial y^k}\Big|_{x=0} =
  \frac12\left(\delta^j_k - a^ja^k\right).
\end{displaymath}
Putting
\begin{displaymath}
  \cos t = \frac{y^1}{|y|}, \quad \sin t = \frac{y^2}{|y|} 
\end{displaymath}
and omitting the constant terms from the last three functions we
obtain that the functions (\ref{indep}) are independent if and only if
the functions
\[1,\quad \cos t,\quad \sin t, \quad \cos t\sin t\left(1 - a^1\cos t -
  a^2\sin t\right) + \left(\cos t - a^1\right) \left(\sin t -
  a^2\right),\] or equivalently, the functions $1$, $\cos t$, $\sin t$,
$\sin 2t (2\!-\!a^1\cos t \!-\! a^2\sin t)$ are linearly
independent. Clearly, this is the case and hence the infinitesimal
holonomy algebra contains $4$ linearly independent vector fields. It
follows from Proposition \ref{sec:prop-1} that $\mathfrak{hol}^{*}(x)$
is infinite-dimensional and hence the holonomy group $\mathsf{Hol}_M(x)$
is not a finite-dimensional Lie group. \ep

\section{The holonomy group of the Funk surface}

\begin{definition} 
  \emph{A Randers surface $(M, \F)$ is called \emph{Funk surface} if
    its Finsler function is defined by
    \begin{equation}
      \label{eq:funk}
      \F(x,y) = \frac{\sqrt{|y|^2 - 
          \left(|x|^2|y|^2 - \langle x,y\rangle^2\right)} + \langle
        x,y\rangle}{1 - |x|^2}
    \end{equation}
    on the unit disk $\mathbb D^2 \!=\! \{x\!\in\!\mathbb{R}^2,
    |x|\!<\!1\}$.}
\end{definition}
We remark, that (\ref{eq:funk}) can be obtained from (\ref{projective})
putting $a=0$.  The Funk surface is a projectively flat Finsler surface
of constant flag curvature $-\frac{1}{4}$.
\\[1ex]
We recall that the Lie algebra ${\mathfrak X}^{\infty}{(S^1)}$ of smooth
vector fields on the circe contains a dense subalgebra known as the
\emph{real Witt algebra} (\cite{FF}, p.164). It consists of vector fields
with finite Fourier series, and hence it is linearly generated by the
vector fields
\begin{displaymath}
  \cos n t \, \frac{\partial}{\partial t}, \quad   
  \sin n t \, \frac{\partial}{\partial t}, \quad n = 0,1,2,\dots
\end{displaymath}
For the Funk surface the indicatrix $\I_0M$ at $x\!=\!0\!\in\! \mathbb
D^2$ is the unit circle $S^1\subset T_0\mathbb D^2$ and we have the
following
\begin{theorem}
  \label{prop:funk}
  The infinitesimal holonomy algebra of the Funk surface $(\mathbb
  D^2, \F)$ at $0\in \mathbb D^2$ contains the real Witt algebra.
\end{theorem}
\begin{proof}
  Let us consider the curvature vector field
  \[\xi\!=\!R \left( \frac{\partial}{\partial
      x_1},\frac{\partial}{\partial x_2} \right)\Big|_{x=0} = -\frac
  14\left(\delta_2^ig_{1m}(0,y)y^m - \delta_1^ig_{2m}(0,y)y^m\right) \]
  Since $\F|_{x=0}= |y|$, we have
  $g_{jm}(0,y)y^m = y^j$ and hence
  \begin{math}
    \xi\!= \!\frac14\! \left[ \begin {array}{r} y^2 \\
        -y^1\end {array} \right].
  \end{math}
  According to (\ref{n1xi}) the first covariant derivatives are:
  \[\nabla_k \xi = 3\frac{\partial
    P}{\partial y^k}\Big|_{x=0}\xi = \frac32\frac{y^k}{|y|}\xi.\] Since
  $\nabla_1 \xi$ and $\nabla_2 \xi$ are not constant multiples of the
  vector field $\xi$, they are not elements of the one-dimensional
  curvature algebra 
  \begin{math}
    \mathfrak R_0=\{c \! \cdot\! \xi \ | \ c \in \mathbb R\}.
  \end{math}
  Let us introduce polar coordinates $y^1 = r\cos t$, $y^2 = r\sin t$ in
  the tangent space $T_0 \mathbb D^2$, then we can express the curvature
  vector field and its first covariant derivatives at $x\!=\!0$ by
  \begin{alignat*}{1}
    \xi \!=\! -\frac{1}{4} \frac{\partial}{\partial t}, \qquad \nabla_1
    \xi \!=\! -\frac{3}{8}\cos t \frac{\partial}{\partial t},\qquad
    \nabla_2 \xi &\!=-\! \frac{3}{8} \sin t\frac{\partial}{\partial t}.
  \end{alignat*}
  Hence the vector fields 
  \[\frac{\partial}{\partial t} = -4\xi,\quad 
  \cos t\frac{\partial}{\partial t} = -\frac83\nabla_1 \xi\quad \sin
  t\frac{\partial}{\partial t} = -\frac83\nabla_2 \xi\] are elements of
  the infinitesimal holonomy algebra $\mathfrak{hol}^{*}(0)$.
  \\[1ex]
  Similarly, the second covariant derivatives of the curvature
  vector field are
  \begin{equation}
    \label{eq:D_2}
    \nabla\!_1\nabla\!_1\xi = 
    -\tfrac{3}{16}(4\cos^2 t+1)\frac{\partial}{\partial t},  \quad
    \nabla\!_2\nabla\!_2\xi =   
    -\tfrac{3}{16}(4\sin^2t +1)\frac{\partial}{\partial t}
  \end{equation}
  and 
  \begin{displaymath}
    \nabla\!_1\nabla\!_2\xi =  \nabla\!_2\nabla\!_1\xi = 
    -\frac34 \cos t\sin t\frac{\partial}{\partial t}.
  \end{displaymath}
  Since $\nabla\!_1\nabla\!_1\xi - \nabla\!_2\nabla\!_2\xi$ and
  $\nabla\!_1\nabla\!_2\xi$ belong to $\mathfrak{hol}^{*}(0)$, we obtain
  that $\cos 2t\frac{\partial}{\partial t}$ and $\sin
  2t\frac{\partial}{\partial t}$ are also elements of
  $\mathfrak{hol}^{*}(0)$.
  \\[1ex]
  Let us suppose now that for $k \!\in\!\mathbb Z$ the vector fields
  $\cos kt\frac{\partial}{\partial t},\;\sin kt\frac{\partial}{\partial
    t}$ belong to $\mathfrak{hol}^{*}(0)$.  We compute
  \begin{displaymath}
    \Big[\cos t\frac{\partial}{\partial t},\cos kt\frac{\partial}{\partial
      t}\Big] - \Big[\sin t\frac{\partial}{\partial t},\sin
    kt\frac{\partial}{\partial t}\Big] = -
    (k-1)\sin(k+1)t\frac{\partial}{\partial t}
  \end{displaymath}
  and
  \begin{displaymath}
    \Big[\cos t\frac{\partial}{\partial t},\sin kt\frac{\partial}{\partial
      t}\Big] + \Big[\sin t\frac{\partial}{\partial t},\cos
    kt\frac{\partial}{\partial t}\Big] = -
    (k+1)\cos(k+1)t\frac{\partial}{\partial t},
  \end{displaymath}
  hence we obtain that $\cos(k+1)t\frac{\partial}{\partial t}$ and
  $\sin(k+1)t\frac{\partial}{\partial t}$ are also elements of
  $\mathfrak{hol}^{*}(0)$, which proves the assertion.
\end{proof}

\begin{theorem}
  The topological closure of the holonomy group of the Funk surface is
  the orientation preserving diffeomorphism group
  $\mathsf{Diff}^{\infty}_{+}(S^1)$.
\end{theorem}
\begin{proof} 
  Since the Funk surface is simply connected, the elements of the
  holonomy group are orientation preserving diffeomorphisms of the
  circle. Hence the topological closure of the holonomy group
  $\overline{\mathsf{Hol}(0)}$ is contained in
  $\mathsf{Diff}^{\infty}_{+}(S^1)$. From the other hand
  $\overline{\mathsf{Hol}(0)}$ contains the exponential image of the
  real Witt algebra. The exponential mapping is continuous (c.f.~Lemma
  4.1 in \cite{Omori2}, p.~79) and the real Witt algebra is dense in
  the Lie algebra of $\mathsf{Diff}^{\infty}(S^1)$, hence
  $\overline{\mathsf{Hol}(0)}$ contains the normal subgroup generated
  the exponential image of the Lie algebra of
  $\mathsf{Diff}^{\infty}(S^1)$. Since $\mathsf{Diff}^{\infty}_+(S^1)$
  is a simple group (cf. \cite{Herman}, Corollaire 2.) we get
  $\overline{\mathsf{Hol}(0)}=\mathsf{Diff}^{\infty}_{+}(S^1)$.
\end{proof}

\end{document}